\documentclass[12pt,a4paper,leqno]{article}

\usepackage[utf8]{inputenc}
\usepackage[T1]{fontenc}
\usepackage[english]{babel}
\usepackage{amsthm}
\usepackage{amsfonts}         
\usepackage{amsmath}
\usepackage{amssymb}
\usepackage{hyperref}
\usepackage{enumerate}
\usepackage{tikz}
\usepackage{mathrsfs}  
\usetikzlibrary{matrix, arrows, decorations.pathmorphing}
\usepackage{hyperref}

\newcommand{\fref}[1]{\hyperref[{#1}]{\ref*{#1}}}

\newcommand{\Eb}{\mathbb{E}}

\newcommand{\Lb}{\mathbb{L}}

\newcommand{\Pb}{\mathbb{P}}

\newcommand{\Tb}{\mathbb{T}}

\newcommand{\Zb}{\mathbb{Z}}

\newcommand{\Ac}{\mathcal{A}}

\newcommand{\Fc}{\mathcal{F}}

\newcommand{\Oc}{\mathcal{O}}

\newcommand{\Uc}{\mathcal{U}}

\newcommand{\Ls}{\mathscr{L}}

\newcommand{\coim}{\mathrm{\coim}}

\newcommand{\Hom}{\mathrm{Hom}}

\newcommand{\Pic}{\mathrm{Pic}}

\newtheorem{theo}{Tplottin ubuntuheorem}[section]
\theoremstyle{plain}
\newtheorem{thm}[theo]{Theorem}
\newtheorem{lem}[theo]{Lemma}
\newtheorem{prop}[theo]{Proposition}

\newtheorem*{thm*}{Theorem}
\newtheorem*{lem*}{Lemma}
\newtheorem*{prop*}{Proposition}
\newtheorem*{cor*}{Corollary}

\theoremstyle{definition}

\theoremstyle{remark}
\newtheorem{rem}[theo]{Remark}

\title{On line bundles in derived algebraic geometry}
\author{Toni Annala}
\date{}
\begin{document}

\maketitle
\begin{abstract}
We give examples of derived schemes $X$ and a line bundle $\Ls$ on the truncation $tX$ so that $\Ls$ does not extend to the original derived scheme $X$. In other words the pullback map $\Pic(X) \to \Pic(tX)$ is not surjective. Our examples have the further property that, while their truncations are projective hypersurfaces, they fail to have any nontrivial line bundles, and hence they are not quasi-projective.
\end{abstract}

\tableofcontents

\section{Introduction}

It is a well known fact (see \cite{KST}), that for an affine derived scheme $X$, the induced pullback map $K^0(X) \to K^0(tX)$ is an isomorphism. However, due to the form that the descent spectral sequence takes, one would expect the $K^0$ of a derived scheme to be different from that of its truncation in general. 

Moreover, a more computable invariant -- the Picard group of $X$ -- is a summand of $K^0(X)$: we have a map $\Pic(X) \to K^0(X)$ that sends a line bundle to its $K$-theory class, and a one sided inverse is induced by the perfect determinant map of Schürg-Toën-Vezzosi \cite{STV} as the determinant of a line bundle $\Ls$, regarded as a perfect complex, is again $\Ls$. We can therefore conclude that if the pullback map $\Pic(X) \to \Pic(Y)$ fails to be injective (surjective), then so does the map $K^0(X) \to K^0(Y)$.

It is not very hard to find examples of derived schemes $X$ so that the map $\Pic(X) \to \Pic(tX)$ is not injective. Some kind of trivial derived enhancement will often have this property: for example, take the derived scheme whose underlying scheme is $\Pb^2$, and whose structure sheaf is given by the trivial square zero extension $\Oc_{\Pb^2} \oplus \Oc_{\Pb^2}(-3)[-1]$. One can compute, either using the descent spectral sequence or the deformation sequences as in \fref{DeformationSequences}, that the Picard group of $X$ is isomorphic to $\Zb \oplus k$. 

However, finding an example so that $\Pic(X) \to \Pic(tX)$ fails to be surjective is harder, as a trivial extension will not work anymore. It is also a much more interesting question. Consider for example the question of whether or not a derived scheme $X$ is quasi-projective. In \cite{An2} it was noted that a derived scheme $X$ is quasi-projective if and only if it has an \emph{ample line bundle}, i.e., a line bundle whose truncation is ample on $tX$. Hence, the question of whether or not $X$ is quasi-projective can be divided into two parts: 
\begin{itemize}
\item is the truncation $tX$ quasi-projective;
\item does an ample line bundle on $tX$ extend to $X$;
\end{itemize}
and the second question is obviously related to the surjectivity of the map $\Pic(X) \to \Pic(tX)$.

The main purpose of this article is finding an example of a derived scheme $X$ such that the pullback map $\Pic(X) \to \Pic(tX)$ is not surjective. However, in the examples we construct, $\Pic(X)$ is trivial while the truncation $tX$ is a projective hypersurface, realizing the second obstruction (as above) to quasi-projectivity. The examples are constructed in Section \fref{Computation}, and verifying the desired properties is an easy computation involving nothing else than just the basic graduate knowledge of algebraic geometry. However, justifying these computations takes a bit more work, and is done in section \fref{DeformationTheory}.

\subsection*{Conventions}

Throughout this article, we are going to work over a field $k$ of characteristic $0$. A \emph{derived ring} (over $k$) will be either a simplicial $k$-algebra, connective differential graded $k$-algebra or a connective $\Eb_\infty$-ring spectrum over $k$. Under our characteristic 0 assumption, all the above models are known to agree. All derived rings will be assumed to be commutative (in homotopical sense). As everything in this paper should be assumed to be  derived, we will often drop the word ''derived'' to not to burden the exposition. We will denote by $[-n]$ the operation of $n$-fold suspension $\Sigma^n$, which in the dg-world corresponds to the homological shift upwards $n$ times. Throughout the article, unless otherwise specified, $X$ will be a derived scheme over $k$. All derived schemes are assumed to be separated.

\subsection*{Acknowledgements}

The author would like to thank his advisor Kalle Karu for useful discussions relating the material of this article back to classical algebraic geometry. The author is supported by the Vilho, Yrjö and Kalle Väisälä Foundation of the Finnish Academy of Science and Letters.

\section{Deformation theory}\label{DeformationTheory}

A natural way to approach the problem is using the derived version of deformation theory. All the main references \cite{TV2}, \cite{GR}, \cite{Lur2} and \cite{Lur3} for derived algebraic geometry deal with the subject in one form or another. Other sources include short notes \cite{PV} made available by Porta and Vezzosi.

\subsection{Background}

Deformation theory of a derived scheme is controlled by its cotangent complex. Let us recall the definition of a derived infinitesimal extension of a (derived) ring. Suppose $A$ is a ring, $B$ is an $n$-truncated $A$-algebra and let $M$ be a $\pi_0(B)$-module. An infinitesimal extension of $B$ by $M[-n-1]$ as an $A$-algebra is an $A$-algebra map $B' \to B$ whose homotopy fibre (as $B'$-modules) is equivalent to $M[-n-1]$. We note that in this case $\pi_i(B') = \pi_i(B)$ when $n \leq n$, $\pi_{n+1}(B') = M$ and $\pi_i(B')=0$ for $i > n+1$. We recall how infinitesimal extensions can be classified using the cotangent complex $\Lb_{B/A}$:

\begin{thm}
The derived infinitesimal extensions of the $A$-algebra $B$ by $M[-n-1]$ are classified by elements of $\pi_0(\Hom_B(\Lb_{B/A}, M[-n-2])) \cong \pi_{-n-2}(\Hom_B(\Lb_{B/A}, M))$.

The relationship is as follows: A map $\Lb_{B/A} \to M[-n-2]$ is exactly the data of an $A$-derivation $d: B \to M[-n-2]$, which in turn is exactly the data of a map of $A$-algebras $(1,d): B \to B \oplus M[-n-2]$ where $B \oplus M[-n-2]$ is the \emph{trivial square zero extension}. Now the infinitesimal extension $B' \to B$ associated to $d$ is given as the pullback
\begin{center}
\begin{tikzpicture}[scale = 1.5]
\node (A2) at (2,2) {$B$};
\node (B2) at (2,1) {$B \oplus M[-n-2]$};
\node (A1) at (0,2) {$B'$};
\node (B1) at (0,1) {$B$};
\path[every node/.style={font=\sffamily\small}]
(A1) edge[->] (A2)
(B1) edge[->] node[below]{$(1,d)$}(B2)
(A1) edge[->] (B1)
(A2) edge[->] node[right]{$(1,0)$} (B2)
;
\end{tikzpicture}
\end{center} 
where $0$ is the trivial derivation.
\end{thm}

In this paper we will be interested in the case when $A = k$ is a field and the $k$-algebra $B$ is $0$-truncated. Moreover, we want to find an extension $B'$ of $B$ to a $k[S^n]$-algebra. In other words, we are interested in all pushout diagrams 
\begin{center}
\begin{tikzpicture}[scale = 1.5]
\node (A2) at (2,2) {$k$};
\node (B2) at (2,1) {$B$};
\node (A1) at (0,2) {$k[S^n]$};
\node (B1) at (0,1) {$B'$};
\path[every node/.style={font=\sffamily\small}]
(A1) edge[->] (A2)
(B1) edge[->] (B2)
(A1) edge[->] (B1)
(A2) edge[->] (B2)
;
\end{tikzpicture}
\end{center} 
Here $k[S^n]$ is the $k$-algebra incarnation of the $n$-sphere, i.e., the trivial square-zero extension $k \oplus k[-n]$. Such extensions of $B'$ of $B$ are exactly the infinitesimal extensions of $B$ by $B[-n]$, and hence they are classified by the cotangent complex. We record this in the following

\begin{prop}
Let $A$ be a classical $k$-algebra. Now the deformations $A'$ of $A$ over $k[S^n]$ are in one to one correspondence with $\pi_0(\Hom(\Lb_{A/k}, A[-n-1])) \cong H^{n+1} (\Tb_{A/k})$, where $\Tb_{A/k}$ is the \emph{tangent complex} of $A$ over $k$. Note that such a deformation $A'$ is necessarily derived, and the map $A' \to A$ is the map to the truncation. 
\end{prop}

The above results have immediate global analogues for schemes.

\subsection{Deformation sequences in derived infinitesimal extensions}\label{DeformationSequences}

Let us have an extension $\Oc_{X'} \to \Oc_X$ defined by
\begin{center}
\begin{tikzpicture}[scale = 1.5]
\node (A2) at (2,2) {$\Oc_X$};
\node (B2) at (2,1) {$\Oc_X \oplus \Fc[-1]$};
\node (A1) at (0,2) {$\Oc_{X'}$};
\node (B1) at (0,1) {$\Oc_X$};
\path[every node/.style={font=\sffamily\small}]
(A1) edge[->] (A2)
(B1) edge[->] node[below]{$(1,d)$}(B2)
(A1) edge[->] (B1)
(A2) edge[->] node[right]{$(1,0)$} (B2)
;
\end{tikzpicture}
\end{center} 
where $\Fc$ is a connective quasi-coherent sheaf over $X$, and $d,0$ are derivations $\Oc_X \to \Fc[-1]$. Note that this includes, but is more general than our definition of an infinitesimal extension (although that will be our only case of interest).

Taking vertical in the above diagram cofibres, we get
\begin{center}
\begin{tikzpicture}[scale = 1.5]
\node (A2) at (2,2) {$\Oc_X \oplus \Fc[-1]$};
\node (B2) at (2,1) {$\Fc[-1]$};
\node (A1) at (0,2) {$\Oc_{X}$};
\node (B1) at (0,1) {$\Fc[-1]$};
\path[every node/.style={font=\sffamily\small}]
(A1) edge[->] node[above]{$(1,d)$} (A2)
(B1) edge[->] (B2)
(A1) edge[->] (B1)
(A2) edge[->] node[right]{$p_2$} (B2)
;
\end{tikzpicture}
\end{center}  
where $p_2$ is the natural projection. Hence the left vertical map, the map that induces the connecting morphism of the long exact sequence of cohomology groups induced by the cofibre sequence
\begin{equation*}
\Fc \to \Oc_{X'} \to \Oc_{X},
\end{equation*}  
is $d$. To understand the behavior of the Picard group in infinitesimal extensions, we need the following slight modification of above.

\begin{prop}
The connecting morphisms $H^{i}(\Oc^*_X) \to H^{i+1}(\Fc)$ in the long exact sequence induced by the cofibre sequence
\begin{equation*}
\Fc \to \Oc_{X'}^* \to \Oc_{X}^*
\end{equation*}
are given by the \emph{log differential} $\delta: \Oc_{X}^* \to \Fc[-1]$ associated to $d$, i.e., the map ''defined'' by the formula $a \mapsto d(a)/a$.
\end{prop}
\begin{proof}
Taking multiplicative groups of units preserves all limits of rings (this is true by definition as $A^*$ is the space of maps from $k[t,t^{-1}]$ to $A$), and hence we get the following diagram in Abelian sheaves
\begin{center}
\begin{tikzpicture}[scale = 1.5]
\node (D2) at (2.5,3) {$\Fc$};
\node (D1) at (0,3) {$\Fc$};
\node (C2) at (2.5,0) {$\Fc[-1]$};
\node (C1) at (0,0) {$\Fc[-1]$};
\node (A2) at (2.5,2) {$\Oc_X^*$};
\node (B2) at (2.5,1) {$(\Oc_X \oplus \Fc[-1])^*$};
\node (A1) at (0,2) {$\Oc_{X'}^*$};
\node (B1) at (0,1) {$\Oc_X^*$};
\path[every node/.style={font=\sffamily\small}]
(A1) edge[->] (A2)
(B1) edge[->] node[below]{$(1,d)$}(B2)
(A1) edge[->] (B1)
(A2) edge[->] node[right]{$(1,0)$} (B2)
(D1) edge[->] (D2)
(D1) edge[->] (A1)
(D2) edge[->] (A2)
(B1) edge[->] node[right]{$\delta$} (C1)
(B2) edge[->] (C2)
(C1) edge[->] (C2)
;
\end{tikzpicture}
\end{center} 
where all the squares are pullback squares, and the columns are distinguished triangles. The formula given in the statement follows trivially from considering the situation in a concrete model (say, simplicial commutative $k$-algebras, where the result is readily seen to hold in every degree).
\end{proof}

In order to utilize the above result, we need to be able to understand how the log derivation $\delta$ acts on the related cohomology groups. The result is most easily stated and proved in terms of Čech cohomology.

\begin{prop}
Suppose $\Fc$ is a quasi-coherent and discrete sheaf on $X$, $\Ac$ a discrete Abelian sheaf on $X$, and let $\delta: \Ac \to \Fc[-n]$ be a map of Abelian sheaves. Let $\Uc = (U_i)_{i \in I}$ be an affine cover of $X$, and suppose $\delta$ is given by the Čech cycle of maps $(\delta_{i_0 \cdots i_n}) $ from $\Ac$ to $\check{C}^{n}_\Uc(\Fc)$.

Now the induced maps $\check{H}_\Uc^j(\Ac) \to H^{n+j}(\Fc)$ can be presented by the formula
\begin{equation*}
(\delta a)_{i_0 \cdots i_{n+j}} = (-1)^n \delta_{i_0 \cdots i_n} (a_{i_n \cdots i_{n+j}})
\end{equation*}
on Čech cycles.
\end{prop}
\begin{proof}
As $\check{C}^*_\Uc(\Ac)$ is weakly equivalent to $\Ac$, one hopes that there would be only one map (up to homotopy) $\check{C}^*_\Uc(\Ac) \to \check{C}^{*+n}_\Uc(\Fc)$ that extends the original map $\Ac \to \check{C}^{n}_\Uc(\Fc)$. If this were true, then it would suffice to verify that the above formula for $\delta$ gives a well defined map of chain complexes $\check{C}^*_\Uc(\Ac) \to \check{C}^{*+n}_\Uc(\Fc)$. This is indeed true, and the easy algebraic manipulation is left for the reader. Now the fact that the formula given for $\delta$ is the right one follows from the fact that the Čech complex maps (quasi-isomorphically) to an injective resolution and that maps to injective resolutions preserve homotopy equivalences. 
\end{proof}

\begin{rem}
Note that the assumptions of the above proposition are satisfied in our case of interest. Indeed, suppose $X$ is a classical scheme, and $d: \Oc_X \to \Fc[-n]$ is a derivation. This can always be presented as a cocycle $\check{C}^n_\Uc(\Fc)$ valued derivation for any affine cover $\Uc$, and hence we also obtain a similar presentation for the induced log derivation $\delta$. 
\end{rem}

Returning to our specific case of interest, let $X$ be a classical $k$-scheme. It is clear from the formula given by the above proposition that the map
\begin{equation*}
H^{i}(T_X) \times H^{j}(\Oc_X^*) \to H^{i+j}(\Oc_X)
\end{equation*}
is $\Zb$-bilinear, and $k$-linear in the first argument. As a special case we obtain an \emph{obstruction pairing}
\begin{equation*}
\{\text{Deformations $X'$ of $X$ over $k[S^n]$}\} \times \Pic(X) \to H^{n+2}(\Oc_X)
\end{equation*}
controlling which line bundles on $X$ extend to which deformations $X'$.

\section{The example}\label{Computation}

In this section we are going to give the example. Let $X \hookrightarrow \Pb^n$ be a smooth hypersurface of degree $n+1$ defined as the vanishing locus of a homogeneous polynomial $F$. Without loss of generality we may assume that $X$ does not contain the point $[1 : 0 : \cdots : 0]$ so that $\Uc = (U_i|_X)^n_{i \geq 1}$, where $(U_i)^n_{i \geq 0}$ is the standard open cover of $\Pb^n$, is an affine cover of $X$.

Computing the Čech cohomology groups of $\Oc_X$ and $T_X$ associated to the above covering, one obtains the following two lemmas. The results are completely elementary, and can be worked out by nothing more than a few pages of diagram chasing. For completeness, however, we give short proofs.

\begin{lem}\label{CohomologyOfStructureSheaf}
Suppose $n \geq 3$. Then the cohomology group $H^{n-1}(\Oc_X)$ is isomorphic to $k$, and it is generated by the Čech cocycle
\begin{equation*}
{\partial_0 F \over x_1 \cdots x_n}.
\end{equation*}
Moreover, $H^{i}(\Oc_X) \cong 0$ for $1 \leq i \leq n-2$.
\end{lem}
\begin{proof}
All the other claims than that the given cocycle generates are standard. The last remaining claim follows from the fact that $\partial_0 F$ has a term $c x_0^n$, where $c \not = 0$, and $x_0^n / (x_1 \cdots x_n)$ is not a boundary (unlike all other possibilities).
\end{proof}

\begin{lem}\label{TheDerivation}
Suppose $n \geq 4$. Then the cohomology group $H^{n-2}(T_X)$ is isomorphic to $k$, and it is generated by the Čech cocycle $d = (d_{12\cdots \hat{i} \cdots n})_{1 \leq i \leq n}$, where
\begin{equation*}
d_{12\cdots \hat{i} \cdots n} = (-1)^i {(\partial_0 F) \partial_i - (\partial_i F)\partial_0 \over x_1 x_2 \cdots \hat{x_i} \cdots x_n}
\end{equation*}
\end{lem}
\begin{rem}
As is customary, we use the hat to denote an index or a term which is left out.
\end{rem}
\begin{proof}
The fact that $H^{n-2}(T_X) \cong k$ follows immediately from Serre duality once we recall that by hard Lefschetz $H^1(X; \Omega^1_X) \cong H^1(\Pb^n; \Omega^1_{\Pb^n}) \cong k$. Moreover, $d$ is a cocycle in derivations on $X$: clearly all the chosen derivations send $F$ to 0, so they are derivations on $X$, and they do satisfy the cocycle condition
\begin{align*}
\sum_{i=1}^n (-1)^i d_{12\cdots \hat{i} \cdots n} &= \sum_{i=1}^n {(\partial_0 F) \partial_i - (\partial_i F)\partial_0 \over x_1 x_2 \cdots \hat{x_i} \cdots x_n} \\
&= \sum_{i=1}^n {x_i (\partial_0 F) \partial_i - x_i (\partial_i F)\partial_0 \over x_1 x_2 \cdots x_n}\\
&= {x_0 (\partial_0 F) \partial_0 - x_0 (\partial_0 F)\partial_0 \over x_1 x_2 \cdots x_n} \\
&=0
\end{align*}
as the Euler form equals $0$. Hence $d$ generates $H^{n-2}(T_X)$ if it is nontrivial, but this in fact follows from the calculation following the lemma.
\end{proof}

We are now ready to show that the line bundle $\Oc_{\Pb^n}(-1)|_X$ does not extend to any nontrivial first order deformation of $X$ over $k[S^{n-3}]$. Recall that the transition maps $\alpha_{ij}$ of $\Oc(-1)$ are defined as $\alpha_{ij} = {x_j \over x_i}$.

We can now just apply the generating log differential to $\Oc_{\Pb^n}(-1)|_X$.
\begin{align*}
(\tilde d \alpha)_{12 \cdots n} &= (-1)^{n-2} \tilde{d}_{12 \cdots n-1}(\alpha_{n-1,n}) \\
&=  (-1)^{n-2} {x_{n-1} \over x_n} (-1)^{n} {(\partial_0 F) \partial_n - (\partial_n F)\partial_0 \over x_1 x_2 \cdots x_{n-1}} \left( x_n \over x_{n-1} \right) \\
&= {x_{n-1} \over x_n} {(\partial_0 F)x_{n-1}^{-1} \over x_1 x_2 \cdots x_{n-1}} \\
&= {(\partial_0 F)\over x_1 x_2 \cdots x_{n}} 
\end{align*}
which is known to be nonzero by an earlier lemma. We have proven the following

\begin{thm}
Let $X \hookrightarrow \Pb^n$, $n \geq 4$, be a smooth projective hypersurface of degree $n+1$. Then the line bundle $\Oc(1)|_X$ does not extend to any nontrivial deformation of $X$ over $k[S^{n-3}]$.
\end{thm}
\begin{proof}
Indeed, as we noticed earlier, the obstruction of a line bundle is $k$-linear in the deformation of $X$. We have shown that the obstruction is not 0 for the generating deformation, and therefore it will not be 0 for any nonzero multiple of it. 
\end{proof}

\begin{rem}
Note that when $n=3$, the derivation $\delta$ given in \fref{TheDerivation} is still a perfectly valid element of $H^1(T_X)$, and applying the log differential to $\Oc(1)|_X$ as above shows that $\Oc(1)|_X$ does not extend to the deformation associated to $\delta$. This has a moduli theoretic interpretation. Indeed, it is a well known fact that the moduli space of \emph{polarized K3 surfaces} (K3 surfaces equipped with an ample line bundle) is 19 dimensional. Therefore the kernel,  which can easily be checked to be 19 dimensional, of the map $H^1(T_X) \to H^2(\Oc_X)$ given by evaluating at $\Oc(1)|_X$ should be thought as the tangent space of the space of polarized $K3$ surfaces, sitting inside the tangent space of the moduli of $K3$ surfaces.
\end{rem}

Assume again that $n \geq 4$. It is known that the Picard group of a smooth hypersurface $X \hookrightarrow \Pb^n$ is isomorphic to $\Zb$ and generated by $\Oc(1) |_X$. Hence

\begin{thm}
Let $X \hookrightarrow \Pb^n$, $n \geq 4$, be a smooth projective hypersurface of degree $n+1$ over an infinite field. Then $\Pic(X) \cong 0$ and therefore $X$ fails to be quasi-projective.
\end{thm}
\begin{proof}
We have the deformation sequence
\begin{equation*}
\cdots \to H^{n-2}(\Oc_X) \to \Pic(X') \to \Pic(X) \stackrel \delta \to H^{n-1}(\Oc_X) \cdots
\end{equation*}
The claim follows from the fact that $\delta$ is injective and $H^{n-2}(\Oc_X)$ is trivial.
\end{proof}

\end{document}